\documentclass[11pt, a4paper, leqno]{amsart}
\usepackage{amsmath}
\usepackage{amsfonts}
\usepackage{amssymb}
\usepackage{graphicx}
\usepackage[all]{xy}
\usepackage{mathrsfs}

\newtheorem{theorem}{Theorem}[section]

\newtheorem{problemz}{Problem}

\newtheorem{definition}[theorem]{Definition}

\newtheorem{example}[theorem]{Example}

\newtheorem{lemma}[theorem]{Lemma}

\newtheorem{prop}[theorem]{Proposition}
\newtheorem{remark}[theorem]{Remark}

\numberwithin{equation}{section}

\newcommand{\Rr}{\mathbb R}

\renewcommand{\d}{\mathrm d}

\newcommand{\J}{\ensuremath{\mathcal{J}}}

\DeclareMathOperator{\Aut}{Aut} 
\DeclareMathOperator{\End}{End}

\DeclareMathOperator{\GL}{GL}

\DeclareMathOperator{\pr}{pr}

\renewcommand{\hom}{\mathrm{Hom}}
\DeclareMathOperator{\modular}{mod}     
\renewcommand{\mod}{\modular}

\newcommand{\al}{\alpha}                
\renewcommand{\gg}{\mathfrak{g}}        
\newcommand{\hh}{\mathfrak{h}}          
\newcommand{\ggl}{\mathfrak{gl}}          
\renewcommand{\Im}{\text{\rm im}\,}     
\newcommand{\Ad}{\text{\rm Ad}\,}       
\newcommand{\ad}{\text{\rm ad}\,}       

\newcommand{\rank}{\text{\rm rank}\,}   
\newcommand{\K}{\mathcal{K}}                     

\title[Stability and Rigidity Results for Lie Algebras]{A Survey on Stability and Rigidity Results for Lie Algebras*}
\thanks{*Research supported by the ERC Starting Grant No. 279729.}

\author{Marius Crainc}
\address{Depart. of Math., Utrecht University, 3508 TA Utrecht, The Netherlands}
\email{m.crainic@uu.nl}

\author{Florian Sch\"atz}
\address{Depart. of Math., Utrecht University, 3508 TA Utrecht, The Netherlands}
\email{florian.schaetz@gmail.com}

\author{Ivan Struchiner}
\address{Depart. of Math., University of S\~ao Paulo, Cidade Universit\'aria, S\~ao Paulo, Brasil}
\email{ivanstru@ime.usp.br}

\date{}

\begin{document}

\maketitle

\begin{abstract}
We give simple and unified proofs of the known stability and rigidity results for Lie algebras, Lie subalgebras and Lie algebra homomorphisms. Moreover,
we investigate when a Lie algebra homomorphism is stable under all automorphisms of the codomain (including outer automorphisms).
\end{abstract}

\section{Introduction}
In these notes we address the following stability/rigidity problems for Lie algebras:

\begin{problemz}[Rigidity of Lie algebras]\rm
\label{problem: rigidity}
Given a Lie bracket $\mu$ on a vector space $\gg$, when is it true that every Lie bracket $\mu'$ sufficiently close to $\mu$ is of the form $\mu' = A\cdot\mu$ for some $A \in \GL(\gg)$ close to the identity?
\end{problemz}

In the problem above, $(A\cdot\mu)(u,v) := A\mu(A^{-1}u, A^{-1}v)$. A Lie algebra which satisfies the condition above will be called \textbf{rigid}.

\begin{problemz}[Rigidity of Lie algebra homomorphisms]\label{problem rigidity homomorphism}\rm
Given a Lie algebra homomorphism $\rho: \hh \to \gg$, when is it true that every Lie algebra homomorphism $\rho':\hh \to\gg$ sufficiently close to $\rho$ is of the form $\rho' = \Ad_g\circ \rho$ for some $g \in G$ close to the identity?
\end{problemz}

A Lie algebra homorphism satisfying the condition above will be called a \textbf{rigid homomorphism}

\begin{problemz}[Rigidity of Lie subalgebras]\label{problem rigidity subalgebra}\rm
Given a Lie subalgebra $\hh \subset \gg$, when is it true that every subalgebra $\hh'\subset \gg$ sufficiently close to $\hh$ is of the form $\hh' = \Ad_g(\hh)$ for some $g \in G$ close to the identity?
\end{problemz}

A subalgebra satisfying the condition above will be called a \textbf{rigid subalgebra}.

\begin{problemz}[Stability of Lie algebra homomorphisms]\label{problem: stability homomorphisms}\rm
Given a Lie algebra homomorphism $\rho: \hh \to \gg$, when is it true that for every Lie algebra $\gg'$ sufficiently close to $\gg$, there exists a homomorphism $\rho':\hh \to\gg'$ close to $\rho$?
\end{problemz}

A homomorphism satisfying the condition above will be called \textbf{stable}.

\begin{problemz}[Stability of Lie subalgebras]\label{problem: stable subalgebra}\rm
Given a Lie subalgebra $\hh \subset \gg$, when is it true that for every Lie algebra $\gg'$ sufficiently close to $\gg$, there exists a Lie subalgebra $\hh' \subset\gg'$ close to $\hh$?
\end{problemz}

A subalgebra satisfying the condition above will be called a \textbf{stable subalgebra}
\\

As a consequence of the formalism used to address these stability problems, we will give an answer also to the following problem:

\begin{problemz}\rm \label{smoothness bracket}
Given a Lie algebra $\gg$, when is a neighborhood of $\gg$ smooth in the space of all Lie algebra structures?
\end{problemz}

Answers to these problems are given in \cite{RN1, RN2, RN3, RN4, RN5}. Our solutions rely on
the following analytic tools:
\begin{enumerate}
 \item A version
of the implicit function theorem (Proposition \ref{thm: IFT}), which guarantees that an orbit of a group action is locally open.
\item A stability result for the zeros of a vector bundle section (Proposition \ref{thm: IFT2}).
\item Kuranishi's description of zero sets.
\end{enumerate}

Our approach allows us to answer Problems 1 to 6 in a simple and unified manner. Moreover the solution of an extension of Problem 4 (Theorem \ref{corol: stable by Aut}) is not present in these papers. In forthcoming work of the authors, infinite-dimensional versions of tool 1. and 2. -- which are Proposition \ref{thm: IFT} and
Proposition \ref{thm: IFT2}, respectively -- will be used to prove stability/rigidity results in the context of Lie algebroids.
\\

As a preparation for answering the problems mentioned above, we first review the infinitesimal deformation theory of Lie algebras, Lie subalgebras, and Lie algebra homomorphisms, respectively.

\section{Lie algebra cohomology}

Infinitesimally, rigidity and stability problems translate into linear algebra problems, which take place in chain complexes associated to Lie algebras and their homorphisms. We briefly recall the construction of these complexes.

Given a Lie algebra $\gg$ and a representation $r:\gg \to \ggl(V)$, one obtains a chain complex called the \textbf{Chevalley-Eilenberg complex of $\gg$ with coefficients in $V$} as follows:
the cochains $C^k(\gg,V)$ of degree $k$ are given by $\hom(\wedge^k\gg, V)$, with differential $\delta_r: C^k(\gg,V) \to C^{k+1}(\gg,V)$ given by
\begin{align*}\delta_r\omega(u_0, \ldots, u_k) := \sum_{i=0}^k (-1)&^ir(u_i)\cdot \omega(u_0, \ldots, \widehat{u_i},\ldots u_k) \\ 
&+ \sum_{i<j}(-1)^{i+j}\omega([u_i,u_j], u_0, \ldots, \widehat{u_i},\ldots, \widehat{u_j}\ldots u_k).
\end{align*}

Whenever necessary to avoid confusion, we will denote a Lie algebra by a pair $(\gg, \mu)$, where $\mu: \wedge^2\gg \to \gg$ denotes the bracket $[\text{ },\text{ }]$ on $\gg$. Also, the differential described above will be denoted by $\delta_{\mu,r}$, when necessary.
For any Lie bracket $\mu$ and representation $r$, the differential satisfies $\delta^2 = 0$ and the resulting cohomology is denoted $H^k(\gg,V)$ (or $H^k_{\mu,r}(\gg,V)$, whenever necessary).

Note that the expression for $\delta_r$ makes sense even when $r:\gg \to \ggl(V)$ is merely a linear map (not necessarily a representation), and when $\mu$ is an arbitrary skew-symmetric bilinear map on $\gg$ to itself (not necessarily satisfying Jacobi identity). In this case, eventhough $\delta^2 \neq 0$, we will continue to denote the corresponding maps by $\delta_{\mu,r}$.

There are three examples that will be extremely important in our treatment of the rigidity/stability problems stated above:

\begin{example}\rm\label{ex: ad}
When $V = \gg$ equipped with the adjoint representation, the differential depends only on the Lie bracket of $\gg$. In this case the differential will be denoted by $\delta = \delta_{\mu}$ whenever there is no risk of confusion, and the cohomology will be denoted by $H^k(\gg,\gg)$.
\end{example}

\begin{example}\rm\label{ex: ad pull-back}
When $\rho: \hh \to \gg$ is a homorphism of Lie algebras, $r = \ad_{\gg}\circ \rho: \hh \to \ggl(\gg)$ is a representation of $\hh$. In this case, the differential will be denoted by $\delta_{\rho} = \delta_{\mu, \rho}$ and the cohomology by $H^k(\hh,\gg)$. 
\end{example}

\begin{example}\rm\label{ex: quotient}
When $\hh \subset \gg$ is a Lie subalgebra, the adjoint representation of $\gg$ induces a representation of $\hh$ on $\gg/\hh$. The corresponding differential will be denoted by $\delta_{\hh} = \delta_{\mu, \hh}$, and the resulting cohomology by $H^k(\hh, \gg/\hh)$.
\end{example}

\section{Deformation Theory: algebraic aspects}

In this section we show the relevance of the cohomologies introduced in Examples \ref{ex: ad}, \ref{ex: ad pull-back}, and \ref{ex: quotient} to the study of deformation problems.

\subsection{Deformations of Lie Brackets}\label{subsection:algebra_deformations_of_brackets}
We begin with the cohomology of $\gg$ with values in its adjoint representation, and its relation to the deformations of the Lie bracket $\mu = [\text{ },\text{ }]$ on the vector space underlying $\gg$.

\begin{definition}
A {\em deformation} of $(\gg, \mu)$ is a smooth one parameter family of Lie brackets $\mu_t$ on $\gg$ such that $\mu_0 = \mu$.

Two deformations $\mu_t$ and $\mu'_t$ are \textbf{equivalent} if there exists a smooth family of Lie algebra isomorphisms 
$\varphi_t: (\gg, \mu_t) \to (\gg, \mu'_t)$
such that $\varphi_0 = \mathrm{id}$.
\end{definition}

\begin{prop}
Let $\mu_t$ be a deformation of $(\gg, \mu)$. Then $\dot{\mu}_0 = \frac{\d}{\d t}|_{t=0}\mu_t$ is a cocycle in $C^2(\gg,\gg)$. Moreover, if $m'_t$ is another deformation which is equivalent to $\mu_t$, then $[\dot{\mu}_0] = [\dot{\mu}'_0]$ in $H^2(\gg,\gg)$.  
\end{prop}

\begin{proof}
Let $\mathcal{J}: \wedge^2\gg^{\ast}\otimes\gg \to \wedge^3\gg^{\ast}\otimes\gg$ denote the Jacobiator
\[\mathcal{J}(\eta)(u,v,w) = \eta(\eta(u,v),w) + \text{ cyclic permutations}.\]
Then, for all $t$ one has that
\[\mathcal{J}(\mu_t) = 0.\]
Differentiating this expressions at $t = 0$ on both sides we obtain
\[\dot{\mu}_0([u,v],w) + [\dot{\mu}_0(u,v),w] + \text{ cyclic permutations } = 0,\]
or in other words, $\delta_\mu(\dot{\mu}_0) = 0$.

Now, assume that $\mu'_t$ is another deformation which is equivalent to $\mu_t$, and let $\varphi_t \in \mathcal{C}^\infty([0,1],\GL(\gg))$ denote the equivalence. We differentiate both sides of
\[\varphi_t(\mu_t(u,v)) =\mu'_t(\varphi_t(u),\varphi_t(v)) \]
at $t=0$ to obtain
\[\dot{\mu}_0(u,v) - \dot{\mu}'_0(u,v) = [u, \dot{\varphi}_0(v)] + [\dot{\varphi}_0(u),v] - \dot{\varphi}_0[u,v].\]
This shows that $\dot{\mu}_0 - \dot{\mu}'_0 = \delta_{\mu} \dot{\varphi}_0$.
\end{proof}

\begin{remark} \rm
One should interpret the previous proposition as stating that, formally, $H^2(\gg,\gg)$ should be identified with the tangent space at $[\mu]$ to the space of Lie brackets on $\gg$ modulo the natural action of $\GL(\gg)$.  In Theorem \ref{corol:  rigidity bracket} we will prove a partial converse to the proposition above, which states that if $H^2(\gg,\gg) = 0$, then $[\mu]$ is an isolated point in the moduli space of Lie brackets on $\gg$.
\end{remark}

In view of the previous remark, a natural question which arises is under which conditions a cocycle $\xi \in C^2(\gg,\gg)$ determines a deformation of $\mu$. Let us denote by $Z^k(\gg,\gg) = \ker \delta_{\mu}$ the set of closed elements of $C^k(\gg,\gg)$, and by $B^k(\gg,\gg)$ the set of those elements, which are coboundaries. We consider the Kuranishi map
\[\Phi: Z^2(\gg,\gg) \to H^3(\gg,\gg), \quad \Phi(\eta) = \mathcal{J}(\eta) \mod B^3(\gg,\gg). \] 

\begin{prop}
If there exists a deformation $\mu_t$ of $\gg$ such that $\dot{\mu}_0 = \xi$, then $\Phi(\xi) = 0$.
\end{prop}

\begin{proof}
Consider the Taylor expansion
\[\mu_t = \mu + t\xi + \frac{1}{2}t^2\eta + o(t^3)\]
of $\mu_t$ around $t = 0$.

Since $\mathcal{J}(\mu_t) = 0$, we obtain
\[0 = \mathcal{J}(\mu_t) = \mathcal{J}(\mu) + t \delta_{\mu}\xi + t^2(\mathcal{J}(\xi) + \delta_{\mu}\eta) + o(t^3), \]
and it follows that $\mathcal{J}(\xi) = - \delta_{\mu}\eta$.
\end{proof}

\begin{remark} \rm
As a consequence of Theorem \ref{corol: regular bracket}, it follows that if $H^3(\gg,\gg)$ vanishes, then every $\xi \in Z^2(\gg,\gg)$ does indeed give rise to a deformation of $\gg$.
\end{remark}

\subsection{Deformations of Homomorphisms}\label{subsection:algebra_homomorphisms}

We now describe the formal deformation theory of Lie algebra homomorphisms. Let $\rho: \hh \to \gg$ be a homomorphism.

\begin{definition}
A \textbf{deformation} of $\rho$ is a smooth family $\rho_t: \hh \to \gg$ of Lie algebra homomorphisms such that $\rho_0 = \rho$.

Two deformations $\rho_t$ and $\rho'_t$ are \textbf{equivalent} if there exists a smooth curve $g_t$ starting at the identity in $G$, such that $\rho'_t = \Ad_{g_t}\circ \rho_t$. Here, $G$ denotes the unique simply connected Lie group which integrates $\gg$.
\end{definition}

\begin{prop}
If $\rho_t$ is a deformation of $\rho$ then $\dot{\rho}_0 = \frac{\d}{\d t}|_{t=0}\rho_t$ is a cocycle in $C^1(\hh,\gg)$. Moreover, if $\rho'_t$ is equivalent to $\rho_t$, then $[\dot{\rho}_0] = [\dot{\rho}'_0]$ in $H^1(\hh,\gg)$.
\end{prop}

\begin{proof}
We differentiate both sides of
\[\rho_t([u,v]) = [\rho_t(u),\rho_t(v)]\]
at $t=0$ to obtain
\[\dot{\rho}_0([u,v]) = [u, \dot{\rho}_0(v)] + [\dot{\rho}_0(u),v].\]
It follows that $\delta_{\rho}\dot{\rho}_0 = 0$.

Now, assume that $\rho'_t(u) = \Ad_{g_t}\circ \rho_t (u)$ for all $u \in \hh$. Differentiating both sides at $t=0$ we obtain
\[\dot{\rho}'_0(u) = \dot{\rho}_0(u) + [\dot{g}_0, \rho(u)],\]
or in other words $\dot{\rho}_0 - \dot{\rho}'_0 = \delta_{\rho}\dot{g}_0$.
\end{proof}

\begin{remark} \rm
Again we interpret this proposition as stating that, formally, $H^1(\hh,\gg)$ can be identified with the tangent space at $[\rho]$ to the space of all Lie algebra homomorphisms from $\hh$ to $\gg$, modulo the adjoint action of $G$. This should be compared to the statement of Theorem \ref{corol: rigidity homomorphism}.
\end{remark}

This remark suggests the following problem: When is a cocycle $\xi \in C^1(\hh,\gg)$ tangent to a deformation of $\rho$? To obtain a partial answer to this question we consider the Kuranishi map
\[\Phi: Z^1(\hh,\gg) \to H^2(\hh,\gg), \quad \Phi(\xi) = \frac{1}{2}[\xi,\xi] \mod B^2(\hh,\gg),\]
where $\frac{1}{2}[\xi,\xi](u,v) = [\xi(u),\xi(v)]_{\gg}$.

\begin{prop}
If there exists a deformation $\rho_t$ of $\rho$ such that $\dot{\rho}_0 = \xi$, then $\Phi(\xi) = 0$.
\end{prop}

\begin{proof}
Consider the Taylor expansion
\[\rho_t = \rho + t\xi +t^2\eta + o(t^3).\]

It follows from $[\rho_t(u),\rho_t(v)] - \rho_t([u,v]) = 0$ that
\begin{align*}0 = [&\rho(u),\rho(v)] - \rho([u,v]) + t\left([\xi(u),\rho(v)] + [\rho(u),\xi(v)] - \xi([u,v])\right) +\\
&+ t^2\left([\rho(u),\eta(v)] + [\eta(u), \rho(v)] +[\xi(u),\xi(v)] - \eta([u,v])\right) + o(t^3)
\end{align*}
and thus, in particular, $\frac{1}{2}[\xi,\xi] = -\delta_{\rho}\eta$.
\end{proof}

\begin{remark} \rm
As a consequence of Theorem \ref{corol: stability homomorphism} it will follow that if $H^2(\hh,\gg) = 0$, then every $\xi \in Z^1(\hh,\gg)$ is indeed tangent to a deformation of $\rho$.
\end{remark}

\subsection{Deformations of Subalgebras}\label{subsection:algebra_subalgebras}

Finally, we describe the infinitesimal deformation theory for Lie subalgebras.

Let $\hh$ be a $k$-dimensional subalgebra of $\gg$ and denote by $\mathrm{Gr}_k(\gg)$ the Grassmannian manifold of $k$-dimensional subspaces of $\gg$.

\begin{definition}
A \textbf{deformation of $\hh$ inside $\gg$} is a smooth curve $\hh_t \in \mathcal{C}^{\infty}([0,1],\mathrm{Gr}_k(\gg))$ such that $\hh_0 = \hh$, and $\hh_t$ is a Lie subalgebra of $\gg$ for all $t$.

Two deformations $\hh_t$ and $\hh'_t$ of $\hh$ inside $\gg$ are said to be \textbf{equivalent} if there exists a smooth curve $g_t$ starting at the identity in $G$ and such that $\hh'_t = \Ad_{g_t}\hh_t$. 
\end{definition}

\begin{prop}
If $\hh_t$ is a deformation of $\hh$ inside of $\gg$, then $\dot{\hh}_0 = \frac{\d}{ \d t}|_{t=0}\hh_t$ is a cocycle in $C^1(\hh, \gg/\hh) = \hh^{\ast}\otimes\gg/\hh$. Moreover, if $\hh_t$ is equivalent to $\hh'_t$, then $[\dot{\hh}_0] = [\dot{\hh}'_0]$ in $H^1(\hh, \gg/\hh)$.  
\end{prop}

\begin{remark} \rm
In the statement of the proposition above we have used the canonical identification of $T_{\hh}\mathrm{Gr}_k(\gg)$ with $\hh^{\ast}\otimes\gg/\hh$ which is obtained as follows: 
%
If $\hh_t$ is a curve in $\mathrm{Gr}_k(\gg)$ starting at $\hh$, we can find a curve $a_t$  in $\GL(\gg)$ starting at the identity and such that $\hh_t = a_t(\hh)$. Then $\dot{\hh}_0$ is represented by
\[\eta \in \hh^{\ast} \otimes \gg/\hh, \quad \eta(u) := \frac{\d}{\d t}\vert_{t=0}a_t(u) \mod \hh.\] 
\end{remark}

\begin{proof}
Let $\hh_t$ be a deformation of $\hh$ inside of $\gg$. As in the remark above, fix a curve $a_t$ in $\GL(\gg)$, starting at the identity, and such that $a_t(\hh) = \hh_t$. Denote by 
\[\bar{a}_t: \gg/\hh \longrightarrow \gg/\hh_t\]
the induced isomorphism. Then, since for each $t$ we have that $\hh_t$ is a Lie subalgebra, we have that
\[\sigma_t \in \wedge^2\hh^{\ast}\otimes\gg/\hh, \quad \sigma_t(u,v) := \bar{a}_t^{-1}([a_t(u),a_t(v)] \mod \hh_t)\]
vanishes identically for all $t \in \Rr$, and all $u,v \in \hh$. Note however, that by definition,
\[\sigma_t(u,v) = \bar{a}_t^{-1}([a_t(u),a_t(v)] \mod \hh_t) = (a_t^{-1}\circ[a_t(u),a_t(v)]) \mod \hh,\]
and thus, by differentiating at $t=0$ we obtain
\[(-\dot{a}_0[u,v] + [\dot{a}_0(u),v] + [u, \dot{a}_0(v)]) \mod \hh = 0.\]
This is just the cocycle condition for $\eta = \dot{a}_0 \mod \hh$.

Next, assume that $\hh'_t = \Ad_{g_t}\hh_t$. Then $a'_t = \Ad_{g_t} \circ a_t$ is a curve in $\GL(\gg)$ which maps $\hh$ to $\hh'_t$. By differentiating both sides of this expression an taking the quotient by $\hh$, we obtain that
\[\dot{a}'_0(u) \mod \hh = (\dot{a}_0(u) + [\al, u]) \mod \hh,\]
where $\al = \dot{g}_0 \in \gg$. This concludes the proof. 
\end{proof}

\begin{remark} \rm
Once more we interpret this proposition as stating that, formally, $H^1(\hh,\gg/\hh)$ can be identified with the tangent space at $[\hh]$ to the space of all Lie subalgebras of dimension $k$ in  $\gg$, modulo the adjoint action of $G$. This should be compared to the statement of Theorem \ref{corol: rigidity subalgebra}.
\end{remark}

This leads us to the following question: Given a cocycle $\xi \in C^1(\hh, \gg/\hh)$, does it induce a deformation of $\hh$ inside of $\gg$?

First, observe that the exact sequence
$$
\xymatrix{
\hh \ar[r] & \gg  \ar[r] & \gg/\hh
}
$$ 
induces a long exact sequence in cohomology
$$ 
\xymatrix{
\cdots \ar[r] & H^1(\hh,\hh) \ar[r] & H^1(\hh,\gg) \ar[r] & H^1(\hh,\gg/\hh) \ar[r] & H^2(\hh,\hh) \ar[r] & \cdots.
}
$$
The connecting homomorphism $H^1(\hh,\gg/\hh) \to H^2(\hh,\hh)$ measures how much an infinitesimal deformation of $\hh$ as a Lie subalgebra 
effects the Lie bracket that $\hh$ inherits from $\gg$. On the level of cocycles, the connecting homomorphism can be realized by choosing a splitting
$$\sigma: \gg/\hh \to \gg$$
and setting
$$ \Omega_{\sigma}: Z^1(\hh,\gg/\hh) \to Z^2(\hh,\hh), \quad \Omega_\sigma(\eta) := \delta( \sigma \circ \eta).$$
We remark that, a priori, $\Omega_\sigma(\eta)$ is a map from $\wedge^2\hh$ to $\gg$. But because $\eta$ is a cocycle, $\Omega_\sigma(\eta)$ is annihilated
by the projection $\gg \to \gg/\hh$, and hence takes values in $\hh$. Observe that, as an element of $\hom(\wedge^2\hh,\hh)$, $\Omega_\sigma(\eta)$ is closed, but not necessarily exact.

We now define
$$\Phi_\sigma: Z^1(\hh,\gg/\hh) \to Z^2(\hh,\gg/\hh)$$
by
$$ \Phi_\sigma(\eta)(u,v) := [(\sigma \circ \eta)(u),(\sigma \circ \eta)(v)] \textrm{ mod } \hh - \eta(\Omega_\sigma(\eta)(u,v)).$$
As above, $\sigma$ denotes a splitting of the short exact sequence $\hh \to \gg \to \gg/\hh$.
We leave it to the reader to check the following two facts:
\begin{enumerate}
\item $\Phi_\sigma(\eta)$ is indeed a cocycle.
\item If one chooses another splitting, say $\sigma'$, $\Phi_\sigma(\eta)$ and $\Phi_{\sigma'}(\eta)$ differ by
$\delta(\eta \circ \mu \circ \eta)$, where $\mu:=\sigma'-\sigma$. 
\end{enumerate}
Hence the Kuranishi map
$$ \Phi: Z^1(\hh,\gg/\hh) \to H^2(\hh,\gg/\hh), \quad \Phi(\eta):= [\Phi_\sigma(\eta)]$$
is well-defined.

\begin{prop}
If there exists a deformation $\hh_t$ of $\hh$ such that $\dot{\hh}_0 = \eta$, then $\Phi(\eta)=0$.
\end{prop}

\begin{proof}
We fix a splitting $\sigma: \gg/\hh \to \gg$ as before. This choice yields a chart
$$\psi_\sigma: \mathrm{Hom}(\hh,\gg/\hh)\to \mathrm{Gr}_k(\gg)$$
around $\hh\in \mathrm{Gr}_k(\gg)$, which associates to each map $\eta: \hh \to \gg/\hh$ the graph of $\eta_\sigma=\sigma\circ \eta$.
Observe that the differential of $\psi_\sigma$ at zero is the identity (here we use the identification $T_\hh\mathrm{Gr}_k(\gg) \cong \mathrm{Hom}(\hh,\gg/\hh)$ explained above).

Suppose we have a deformation $\hh_t$ of $\hh$. Using the chart $\psi_\sigma$, $\hh_t$ yields
a family $\eta_t \in \mathcal{C}^\infty([0,1],\mathrm{Hom}(\hh,\gg/\hh))$, which we expand up to second order in $t$, i.e.
$$ \eta_t \sim t\eta + t^2 \rho + \cdots.$$
We know that the graph of $\eta_t$ is a Lie subalgebra for all $t$,
i.e. for arbitrary $x, y\in \hh$,
$$ [x + \eta_t(x),y + \eta_t(y)]$$
is again an element of the graph of $\eta_t$.
If one expands this condition in powers of $t$, one obtains the following requirements:
\begin{itemize}
 \item $t^0$: $\hh$ is a Lie subalgebra,
 \item $t^1$: $\eta$ is a cocycle of $\mathrm{Hom}(\hh,\gg/\hh)$,
 \item $t^2$: The cocycle $\Phi_\sigma(\eta): \wedge^2\hh \to \gg/\hh$ which represents the cohomology class $\Phi(\eta)$ is equal to $-\delta \rho$, hence $\Phi(\eta) = 0$.
\end{itemize}
In the last item, we used that $\Phi_\sigma(\eta)$ can also be written as
$$ \Phi_\sigma(\eta)(u,v) := [(\sigma \circ \eta)(u),(\sigma\circ \eta)(v)] - (\eta \circ \pi_{\hh}^{\sigma})\left([(\sigma\circ \eta)(u),v] + [u,(\sigma \circ \eta)(v)] \right),$$
where $\pi_{\hh}^\sigma$ is the projection $\gg \to \hh$ induced by the splitting $\sigma$, i.e. $\pi_{\hh}^{\sigma}= \mathrm{id} - \sigma \circ \pi_{\gg/\hh}$.

\end{proof}

\section{Analytic tools}

\subsection{Openess of orbits}

Let $E$ be a vector bundle over $M$. Suppose a Lie group $G$ acts on $E$ in a smooth fashion. We will always assume that
the action preserves the zero-section $Z: M\hookrightarrow E$. It follows that $M$ inherits a $G$-action.

\begin{definition}
A section $\sigma: M\to E$ is called {\em equivariant} if 
$$ \sigma(g\cdot x) = g\cdot \sigma(x)$$
holds for all $g\in G$ and $x \in M$.
\end{definition}

Observe that the zero section $Z: M\to E$ is always equivariant. Also, notice that the zero
set of any equivariant section is mapped into itself under the $G$-action.

\begin{definition}
Let $\sigma$ be an equivariant section of the vector bundle $E\to M$.
A zero $x\in M$ of $\sigma$ is called {\em non-degenerate} if the sequence
$$
\xymatrix{
\mathfrak{g} \ar[r]^-{\d\mu_x} & T_xM \ar[r]^{\d^\mathrm{vert}\sigma}& E_x
}
$$
is exact, where:
\begin{itemize}
 \item $\mathfrak{g}$ is the Lie algebra of $G$, seen as the tangent space of $G$ at the identity.
\item $\mu_x: G\to M$ is the map $\mu_x(g):= g\cdot x$ and $\d\mu_x$ denotes the tangent map from $\mathfrak{g}$ to $T_xM$.
 \item $\d^\mathrm{vert}\sigma$ is the vertial derivative of $\sigma$ at $x$, which is defined as the composition of the usual differential $\d_x\sigma:T_xM \to T_{0_x}E$ with the canonical projection $T_{0_x}E \cong T_xM \oplus E_x \to E_x$.
\end{itemize}
\end{definition}

\begin{prop}\label{thm: IFT}
Suppose $\sigma$ is an equivariant section of the vector bundle $E\to M$
and let $x$ be a non-degenerate zero of $\sigma$.
Then there is an open neighborhood $U$ of $x$ and a smooth map $h: U\to G$ such that
for all $y\in U$ with $\sigma(y) = 0$, one has $h(y)\cdot x = y$.
In particular, the orbit of $x$ under the $G$-action contains an open neighborhood of $x$ inside
the zero set of $\sigma$.
\end{prop}

\begin{proof}
To simplify the notation, we denote the map $\mu_x(g):=g\cdot x$ by $\alpha$.
Observe that by restricting to an open neighborhood of $x$, we can assume $E$ to be trivial, i.e. $E = M\times V$.
The section $\sigma$ then becomes a map $\beta: M\to V$. The vertical differential of $\sigma$ at $x$ 
translates to the usual differential of $\beta$.

Recall that the map which associates to each point of $G$ (respectively $M$) the rank of the differential of $\alpha$ (respectively $\beta$) is lower semicontinuous, and thus it follows that
\[\rank(\d_{g'}\alpha) \geq \rank(\d_\mathrm{id} \alpha) \qquad \text{for all } g' \in G \text{ close enough to } \mathrm{id}\]
and
\[\rank(\d_{x'}\beta) \geq \rank(\d_x \beta) \qquad \text{for all } x' \in M \text{ close enough to } x.\]
On the other hand, the assumption that $\beta \circ \alpha$ is zero implies $\Im \d_{g'}\alpha \subseteq \ker \d_{\alpha(g')}\beta$ and thus
\[\rank(\d_{\alpha(g')}\beta) + \rank(\d_{g'}\alpha) \leq \dim M = \rank(\d_{\alpha(g')}\beta) + \dim \ker(\d_{\alpha(g')}\beta) \]
for all $g'$ sufficiently close to $\mathrm{id} \in G$.
Finally, from $\ker \d_{x}\beta = \Im \d_{\mathrm{id}} \alpha$ and the inequalities above, we conclude that 
\begin{align*}
\rank(\d_{\alpha(g')}\beta) =  \rank(\d_{x}\beta) \qquad \text{and} \qquad \rank(\d_{g'}\alpha) = \rank(\d_{\mathrm{id}}\alpha)
\end{align*}
for all $g' \in G$ close enough to $\mathrm{id}$.
Thus, there exists a neighborhood $W$ of $\mathrm{id} \in G$ such that $\alpha$ has constant rank on $W$ and $\beta|_{\Im \alpha}$ has constant rank on $\alpha(W)$.

We set $m := \dim G$, $r := \rank(\d_\mathrm{id} \alpha)$, and $s := \rank(\d_{x} \beta)$, and use the constant rank theorem to identify locally $G = \Rr^r \times \Rr^{m-r}$ and $M = \Rr^r \times \Rr^s$ in such a way that
\[\alpha(y,z) = (y,0) \quad \text{ and } \quad \beta(y,0) = 0.\]

Now consider the map
\[\psi: \Rr^r \times \Rr^s \to \Rr^r \times V, \quad \psi(y,w) = (y, \beta(y,w)).\]
We note that it has constant (maximal) rank equal to $r + s$. In fact, we already know that the rank of $\d \beta|_{(y,0) = \Im \alpha}$ is  $s$, and since $\frac{\partial \beta}{\partial y^i}(y,0) = 0$ it follows that the matrix $(\frac{\partial \beta^l}{\partial w^j})$ must be of maximal rank ($ = s$). Thus, from the Implicit Function Theorem it follows that $\psi(y,w) = (y,0)$ implies $w = 0$.

Let $h: U \subset M \simeq \Rr^r \times \Rr^s \to G \simeq \Rr^r \times \Rr^{m-r}$ be given by $h(y,w) = (y,0)$. Then for $\beta(x') = 0$ we have that $x' = (y,0)$, and thus $(\alpha\circ h)(y') = y'$. 
\end{proof}


\subsection{Stability of Zeros}

\begin{prop}[Stability of Zeros]\label{thm: IFT2}
Let $E$ and $F$ be vector bundles over $M$. Let $\sigma \in \Gamma(E)$ be a section and $\phi \in \Gamma(\hom(E,F))$ a vector bundle map satisfying $\phi \circ \sigma = 0$. Suppose that $x \in M$ is such that $\sigma(x) = 0$, and 
\begin{equation}\label{eq: exact}
\xymatrix{T_xM \ar[r]^{\d^{\mathrm{vert}}_x\sigma}& E_x\ar[r]^{\phi_x} & F}
\end{equation}
is exact.

Then the following statements hold true:
\begin{enumerate}
\item $\sigma^{-1}(0)$ is locally a manifold around $x$ of dimension equal to the dimension of $\ker(\d_x^{\mathrm{vert}}\sigma)$.

\item If $\sigma'$ is another section of $\Gamma(E)$ which is $\mathcal{C}^0$-close to $\sigma$, and $\phi'$ is another vector bundle map $E \to F$ which is $\mathcal{C}^0$-close to $\phi$ and such that $\phi'\circ\sigma' = 0$, then there exists $x' \in M$ close to $x$ such that $\sigma'(x') = 0$.

\item If moreover $\sigma'$ is $\mathcal{C}^1$-close to $\sigma$, then $\sigma'^{-1}(0)$ is also locally a manifold around $x'$ of the same dimension as $\sigma^{-1}(0)$.
\end{enumerate}
\end{prop}

\begin{proof}
First of all, since the statements in the theorem are all local, we can assume that both $E$ and $F$ are trivial, and thus $\sigma$ and $\phi$ are just smooth maps
\[\sigma: M \to V, \quad \text{and} \quad \phi: M \to \hom(V,W)\] 
for vector spaces $V$ and $W$. Note that in this local description the vertical derivative of the section $\sigma$ at $x$ becomes the usual derivative of $\sigma$ as a map $M \to V$.

Let $A = \Im\d_x\sigma$, and choose a complement $B$ to $A$ in $V$, i.e.,
\[V = A\oplus B.\]
By exactness of \eqref{eq: exact}, $\phi(x)|_B$ is injective, and $\sigma$ intersects $B$ transversely at $\sigma(x) = 0$.
Since injectivity of $\phi(x)|_B$ is an open condition (it can be expressed as a minor of $\phi(x)$ having non-zero determinant), it follows that for every $\phi': M \to \hom(V,W)$ close to $\phi$, and for every $x'$ close to $x$, the map $\phi'(x')|_B$ is injective.

Note that this already implies the first statement. In fact, if $x'$ is close enough to $x$, and $\sigma(x') \in B$, then $\sigma(x') = 0$ because $\phi(x')(\sigma(x')) = 0$, and $\phi(x')$ is injective on $B$. Thus, close to $x$ we have that $\sigma^{-1}(0) = \sigma^{-1}(B)$, and the smoothness follows from the fact that $\sigma$ is transverse to $B$.

Next, we claim that if $\sigma'$ is $\mathcal{C}^0$-close to $\sigma$, then there exists an $x'$ close to $x$ for which $\sigma'(x') \in B$. In order to see this, let us decompose $\sigma$ in 
\[\sigma_A: M \to A, \quad \text{ and } \quad \sigma_B:M \to B.\]
Then, by the definition of $A$, $\sigma_A$ is a submersion, and thus we may choose coordinates on $M$ and on $A$ such that $\sigma_A(u, v) = u$. If $\sigma' = (\sigma'_A,\sigma'_B)$ is close to $\sigma$, then locally $\sigma_A'$ can be written as 
\[\sigma'_A (u,v)= u + f(u,v)\]
for some $f: M \to A$ such that $||f(u,v)|| < \varepsilon$ for all $(u,v) \in M$, and we must show that there exists $(u,v)$ close to $(0,0)$ such that $\sigma'_A(u,v) = 0$. 

Assume that such a pair $(u,v)$ does not exist. Then, for each $v$, we get a well defined map
\[g_v: \mathrm{S}^{a-1} \longrightarrow \mathrm{S}^{a-1}, \quad u \longmapsto  \frac{u + f(u,v)}{||u + f(u,v)||},\]
where $\mathrm{S}^{a-1}$ denotes the unit sphere of dimension $a-1 = \dim A -1$.  

On the one hand, $g_v$ has degree one, since it is homotopic to the identity map through the homotopy
\[g_v(t, u) = \frac{u +t f(u,v)}{||u + tf(u,v)||}.\]

On the other hand, $g_v$ has degree zero since it is homotopic to a constant map via
\[g_v(t, u) = \frac{tu + f(tu,v)}{||tu + f(tu,v)||},\]
which is well defined because we assumed that $\sigma'_A$ does not vanish at any point. 

Thus we have obtained a contradiction and it follows that we can find $x'$ close to $x$ for which $\sigma'(x') \in B$. This immediately implies the second statement in the theorem, for if $\sigma'(x') \in B$, it follows from the injectivity of $\phi'(x')|_B$, and from $\phi'(x')\circ\sigma'(x') = 0$ that $\sigma'(x') = 0$.

Finally, in order to prove the last statement, we note that if $\sigma'$ is $\mathcal{C}^1$-close to $\sigma$, then it also intersects $B$ transversely at $\sigma'(x') = 0$, and thus we can apply again the same argument used to prove the first statement of the theorem.
\end{proof}

\subsection{Kuranishi's description of zero sets}

We reconsider the setting of the previous subsection:
Let $E$ and $F$ be vector bundles over $M$, $\sigma \in \Gamma(E)$. Suppose $\phi \in \Gamma(\hom(E,F))$ is a vector bundle map from $E$ to $F$ satisfying $\phi \circ \sigma = 0$. 
In addition, we assume that a Lie group $G$ acts on $E$ in such a way that 1) the zero section $Z: M\to E$ $\sigma$
is preserved and 2) the section $\sigma$ is equivariant.

\begin{prop}[Kuranishi models]\label{theorem:Kuranishi}
Suppose that $x$ is a zero of the section $\sigma$. Let $\mu_x: G\to M$ be the map $\mu_x(g):= g\cdot x$.

Then there is
\begin{enumerate}
 \item a submanifold $S$ of $M$, containing $x$ and of dimension $\dim \ker(\d_x^\mathrm{vert}\sigma) - \dim \Im(d\mu_x)$,
\item a subbundle $\tilde{E}$ of $E|_S$ of rank $\dim \ker(\phi_x) - \dim \Im(d_x^\mathrm{vert}\sigma)$,
 \item a section $\tilde{\sigma}$ of $\tilde{E}$,
\end{enumerate}
such that the following conditions are satisfied:
\begin{enumerate}
 \item The zero sets of $\tilde{\sigma}$ and of $\sigma|_S$ coincide.
 \item If $x'$ is a zero of $\sigma$ sufficiently close to $x$, then $x'$ lies in the $G$-orbit of a zero of $\tilde{\sigma}$.
 \item The section $\tilde{\sigma}$ vanishes at $x$ to second order.
\end{enumerate}
\end{prop}

In the course of the proof of Proposition \ref{theorem:Kuranishi}, we will use the following lemma:

\begin{lemma}\label{lemma:bad_directions}
Let $f: \mathbb{R}^m \to \mathbb{R}^n$ be a smooth function that maps zero to zero.
Then there is 
\begin{enumerate}
 \item a submanifold $Z$ of $\mathbb{R}^m$ of dimension $\dim \ker(d_0f)$,
 \item a linear subspace
$C \subset \mathbb{R}^n$ of dimension $n-\dim \Im(d_0f)$
\item and a smooth map $\tilde{f}: Z\to C$
\end{enumerate}
such that
\begin{enumerate}
 \item The zero sets of $\tilde{f}$ and of $f$ coincide in a neighborhood of $0$.
 \item The map $\tilde{f}$ vanishes at $0$ to second order.
\end{enumerate}
\end{lemma}

\begin{proof}
Choose a direct sum decomposition $\mathbb{R}^n = \Im \d_0f \oplus C$. Since $f$ is transverse to $C$ at $0$, 
there is a neighborhood $U$ of $0\in \mathbb{R}^m$ such that $f^{-1}(C)\cap U$ is a smooth submanifold.
One now finishes the proof by defining $Z$ to be this submanifold and $\tilde{f}$ to be the restriction of $f$ to $Z$.
\end{proof}

\begin{remark} \rm
Observe that, although $Z$ and $\tilde{f}$ do depend on the choice of a complement $C$ to $\d_0f$, one can extract the following invariants:
\begin{itemize}
 \item The tangent space of $Z$ at $0$ is canonically isomorphic to $\ker \d_0 f$.
 \item Given $v\in \ker \d_0f$, we choose a curve $\gamma$, which starts at $0$ in the direction of $v$.
We associate to it the vector
 $$\lim_{t\to 0} \frac{f(\gamma(t))}{t^2} \quad \in \quad \mathbb{R}^n.$$
 The class of this vector in $\mathbb{R}^n / \Im \d_0 f$ is independent of the choice of $\gamma$. Hence we obtain a quadratic form on $T_0Z$ with values in $\mathbb{R}^n/ \Im d_0 f$, which coincides with the second jet of $\tilde{f}$ at $0\in Z$.
\end{itemize}

\end{remark}

We now prove Porposition \ref{theorem:Kuranishi}:

\begin{proof}
We first fix a transversal $\tau$ to the orbit of $x$ under $G$, i.e. $T_x\tau \oplus T_x(G\cdot x) = T_x M$.
Since
$$G \times \tau \to M, \quad (g,y) \mapsto g\cdot y $$
is submersive at $(x,\mathrm{id})$, the orbit of $\tau$ under $G$ contains an open neighborhood of $x$.
From now on, we work over $\tau$ and replace $E$, $F$, $\sigma$ and $\phi$ by their restrictions to $\tau$.
We fix local trivializations of $E$ and $F$ near $x$, i.e.
$E = \tau \times V$ and $F = \tau \times W$. Then $\sigma$ and $\phi$ correspond to smooth maps $f: \tau \to V$
and $g: \tau \to \mathrm{Hom}(V,W)$, respectively.

Let $K$ be a complement to the image of $g_x: V\to W$. For $x'$ sufficiently close to $x$, $g_x$ will be transverse
to $K$ as well, hence $g_{x'}^{-1}(K)$ forms a subbundle $\hat{E}$ of $E$ if we let $x'$ vary in a neighborhood of $x$.
By $\phi \circ \sigma = 0$, we know that $\sigma$ actually takes values in $\hat{E}$.
We trivialize $\hat{E}$ in a neighborhood of $x$, i.e. $\hat{E}\cong \tau \times \hat{V}$
and obtain a smooth map $\hat{f}:\tau \to \tilde{V}$ which encodes $\sigma$ (as a section of $\hat{E}$).

Applying Lemma \ref{lemma:bad_directions} to $\hat{f}: M \to \hat{V}$
yields a submanifold $S$ of $\tau$ and a function $\tilde{f}: S\to \tilde{V}$
such that the zero sets of $\tilde{f}$ and $\hat{f}$ coincide in a neighborhood of $x$
and with $\tilde{f}$ vanishing to second order at $x$. We define $\tilde{\sigma}$
to be the section corresponding to $\tilde{f}$.

Now suppose $x'$ is a zero of the original map $f$ (and hence of the section $\sigma$). If $x'$ is sufficiently close to $x$, it lies in the orbit
of an element $y\in \tau$. Since $\sigma$ is equivariant, $y$ is also a zero of $f$ and hence of $\hat{f}=f|_\tau$.
Since the zero sets of $\hat{f}$ and $\tilde{f}$ coincide in a neighborhood of $x$,
it follows that $x'$ lies in the orbit of a zero of $\tilde{f}$.

We leave the verification of the dimension of $S$ and of the rank of $\tilde{E}:= \tau \times \tilde{V} $ to the reader.
\end{proof}

\begin{remark} \rm
As in the local case, the construction of $S$, $\tilde{E}$ and $\tilde{\sigma}$ depends on auxiliary choices.
Nevertheless, the following data are invariant:
\begin{enumerate}
 \item The tangent space of $S$ at $x$ is canonically isomorphic to $\ker \d_x^\mathrm{vert}\sigma / \Im \d\mu_x$.
 \item The fibre of $\tilde{E}$ is canonically isomorphic to $\ker \phi_x / \Im \d_x^\mathrm{vert}\sigma$.
 \item There is a quadratic form on $T_xS$ with values in $\tilde{E}_x$, which encodes the second jet of $\tilde{\sigma}$ at $x$.
\end{enumerate}

\end{remark}

\section{Solutions to the Problems}
\subsection{Lie Brackets}

Proposition \ref{theorem:Kuranishi} yields a local description of the moduli space of Lie brackets on $\gg$ in a neighborhood of a given one:

\begin{theorem}\label{theorem: Kuranishi_Lie_brackets}
Let $(\gg,[\cdot,\cdot])$ be a Lie algebra.
There is an open neighborhood $U$ of $0 \in H^2(\gg,\gg)$
and a smooth map $\Phi: U \to H^3(\gg,\gg)$ with the following properties:
\begin{enumerate}
 \item $\Phi$ and its derivative vanish at the origin.
 \item The zeros of $\Phi$ parametrize Lie brackets on $\gg$, with $0$ corresponding to the original bracket $[\cdot,\cdot]$.
 \item The union of the $\GL(\gg)$-orbits passing through Lie brackets parametrized by $\Phi$ contains an open 
neighborhood of $[\cdot,\cdot]$ in the space of all Lie brackets on $\gg$ (topologized as a subset of $\hom(\wedge^2\gg,\gg)$).
\end{enumerate}
\end{theorem}

\begin{proof}
We apply Proposition \ref{theorem:Kuranishi} as follows: First,
note that $\GL(\gg)$ acts naturally on $M = \hom(\wedge^2\gg,\gg)$. This action extends
to the trivial vector bundles over $M$ with fibers $\hom(\wedge^3\gg,\gg)$ and $\hom(\wedge^4\gg,\gg)$, which we denote by $E$ and $F$,
respectively. For the section $\sigma: M \to E$ of Proposition \ref{theorem:Kuranishi} we take the Jacobiator
\[\mathcal{J}(\eta)(u,v,w) = \eta(\eta(u,v),w) + \text{ cyclic permutations}.\]
Observe that the moduli space of interest is $\J^{-1}(0)/\GL(\gg)$.

For the vector bundle morphism $\phi: E\to F$ consider
\[ \phi(\eta) = \delta_\eta: \hom(\wedge^3\gg,\gg) \to \hom(\wedge^4\gg,\gg),\]
where we use the formula for the Lie algebra differential with
$\eta$ in place of the bracket $[\cdot,\cdot]$ and with the map
$$r(u):= \eta(u,\cdot): \gg \to \gg$$
in the place of the representation.

A simple computation shows that $\d^{\mathrm{vert}}_{\mu}\J = \delta_{\mu}$, which concludes the proof.
\end{proof}

Theorem \ref{theorem: Kuranishi_Lie_brackets} was established by Nijenhuis and Richardson in \cite{RN2} (Theorem B).

\begin{remark} \rm
\hspace{0cm}
\begin{enumerate}
 \item The second jet of $\Phi$ at $0$ is given by the Kuranishi map from Subsection \ref{subsection:algebra_deformations_of_brackets}.\
 \item One can also apply Theorem \ref{theorem: Kuranishi_Lie_brackets} to solve Problem \ref{problem: rigidity} (Rigidity of Lie algebras). One obtains that if $H^2(\gg,\gg) = 0$, then $(\gg, \mu)$ is rigid. Moreover, one can apply Proposition \ref{thm: IFT}, which yields a smooth map that associates to each Lie bracket $\mu'$ close to $\mu$ an isomorphism $A \in \GL(\gg)$, such that $A\cdot \mu = \mu'$:
\end{enumerate}
\end{remark}

\begin{theorem}[Rigidity of Lie Brackets] \label{corol: rigidity bracket}
Let $(\gg, \mu)$ be a Lie algebra. If $H^2(\gg,\gg) = 0$, then there exists an open neighborhood $U \subset \wedge^2\gg^{\ast}\otimes\gg$ of $\mu$, and a smooth map $h: U \to \GL(\gg)$ such that $h(\mu')\cdot \mu = \mu'$ for every Lie bracket $\mu' \in U$. In particular,  $(\gg,\mu)$ is rigid. 
\end{theorem}

\begin{proof}
Again we denote by $M = \hom(\wedge^2\gg,\gg)$, by $E$ the trivial vector bundle with fibers $\hom(\wedge^3\gg,\gg)$ and by $\J$ the section of $E$ given by the Jacobiator. We consider also the natural action of $\GL(\gg)$ on $\hom(\wedge^2\gg,\gg)$. The rigidity problem can be reformulated as \emph{``when is it true that for every $\mu' \in \tilde{M}$ close to $\mu$, and such that $\J(\mu') = 0$, there exists $A \in \GL(\gg)$ close the identity map such that $A\cdot\mu = \mu'$?"}. 

We apply Proposition \ref{thm: IFT}. A simple computation shows that
a Lie bracket $\mu$ is a non-degenerate zero of $\mathcal{J}$ iff the sequence
\[\xymatrix{
\hom(\gg,\gg) \ar[r]^-{\delta_{\mu}} & \hom(\wedge^2\gg,\gg) \ar[r]^-{\delta_{\mu}} & \hom(\wedge^3\gg,\gg)},\]
is exact, i.e. iff $H^2(\gg,\gg) = 0$.
\end{proof}

Finally, by applying Proposition \ref{thm: IFT2} we can solve Problem \ref{smoothness bracket}, i.e., we obtain a local smoothness result for the space of Lie brackets:

\begin{theorem} \label{corol: regular bracket}
Let $(\gg, \mu)$ be a Lie algebra. If $H^3(\gg,\gg)=0$ then the space of Lie algebra structures on $\gg$ is a manifold in a neighborhood of $\mu$ of dimension equal to the dimension of $Z^2(\gg,\gg)$.
\end{theorem}

\begin{remark} \rm
We observe that as a consequence of this theorem, one obtains that if $H^3(\gg,\gg) = 0$, then for every $\xi \in Z^2(\gg,\gg)$ there exists a smooth family of Lie bracket $\mu_t$ on $\gg$ such that $[\frac{\d}{\d t}\big{|}_{t = 0}\mu_t] = \xi$.
\end{remark}

\begin{proof}
Let $M = \hom(\wedge^2\gg,\gg)$, $E = M \times \hom(\wedge^3\gg,\gg)$, $F = M \times \hom(\wedge^4\gg,\gg)$ and consider the section $\J \in \Gamma(E)$ and the vector bundle map $\mu \mapsto \delta_{\mu} \in \Gamma(\hom(E,F))$. Then $\delta_{\mu}(\J(\mu)) = 0$ for all $\mu \in M$, and by differentiating at $\mu$ one obtains
\[\xymatrix{\hom(\wedge^2\gg,\gg) \ar[r]^-{\delta_{\mu}}& \hom(\wedge^3\gg,\gg)\ar[r]^-{\delta_{\mu}} & \hom(\wedge^4\gg,\gg),}\]
which is exact if and only if $H^3(\gg,\gg) = 0$.

Thus, in this case, by the first statement of Proposition \ref{thm: IFT2}, one obtains that the space of Lie algebra structures is smooth in a neighborhood of $\mu$.
\end{proof}

\begin{remark} \rm
Observe that Theorem \ref{corol: regular bracket} can also be obtained from Proposition \ref{theorem:Kuranishi}: Reconsider the setting of the proof of
Theorem \ref{theorem: Kuranishi_Lie_brackets} with the $\GL(\gg)$-action replaced by the trivial group action.
The corresponding Kuranishi model consists of a smooth map from an open neighborhood of $Z^2(\gg,\gg)$
to $H^3(\gg,\gg)$. If the latter space is trivial, it follows that the space of Lie brackets is locally a manifold modeled on the vector space $Z^2(\gg,\gg)$.
\end{remark}

\subsection{Homomorphisms of Lie Algebras}
We use Proposition \ref{theorem:Kuranishi} to obtain a local description of the moduli space of Lie algebra homomorphisms.

\begin{theorem}\label{theorem:Kuranishi_homomorphisms}
Let $\rho: \hh \to \gg$ be a Lie algebra homomorophism.
There is an open neighborhood $U$ of $0\in H^1(\hh,\gg)$ and a smooth map
$\Phi: U\to H^2(\hh,\gg)$ with the following properties:
\begin{enumerate}
 \item $\Phi$ and its derivative vanish at the origin.
 \item The zeros of $\Phi$ parametrize Lie algebra homomorphisms from $\hh$ to $\gg$, with $0$ corresponding to $\rho$.
 \item The union of $G$-orbits passing through homomorphisms parametrized by $\Phi$ contains an open neighborhood of $\rho$
in the space of all Lie algebra homomorphisms (topologized as a subset of $\hom(\hh,\gg)$).
\end{enumerate}
\end{theorem}

\begin{proof}
Let $M = \hom(\hh,\gg)$, and $E$ and $F$ be the trivial vector bundles over $M$ with fiber $\hom(\wedge^2\hh,\gg)$ and $\hom(\wedge^3\hh,\gg)$, respectively. Let $G$ be a Lie group which integrates $\gg$, and consider the adjoint action of $G$ on $M$ and $E$.

Let $\K \in\Gamma(E)$ be given by
\[\K(\varphi)(u,v) = [\varphi(u),\varphi(v)]_\gg - \varphi([u,v]_\hh),\]
and note that the zero set of $\K$ is the space of all Lie algebra homomorphisms from $\hh$ to $\gg$. Thus, the moduli space of interest is $\K^{-1}(0)/G$. Consider also the bundle map
\[\phi: E \to F, \quad \phi(\varphi)(\eta) = \delta_{\varphi}(\eta),\]
which makes sense even if $\varphi$ is not a homorphism.

A straightforward computation shows that 
\[\phi(\varphi)(\K(\varphi)) = 0 \text{ for all } \varphi \in M.\]
Moreover, one has that $\d^{\mathrm{vert}}_{\rho}\K = \delta_{\rho}$. 
Thus, a direct application of Proposition \ref{theorem:Kuranishi} concludes the proof.
\end{proof}

\begin{remark} \rm
\hspace{0cm}
\begin{enumerate}
 \item The second jet of $\Phi$ at $0$ is given by the Kuranishi map from Subsection \ref{subsection:algebra_homomorphisms}.
\item One can apply Theorem \ref {theorem:Kuranishi_homomorphisms} to solve Problem \ref{problem rigidity homomorphism} (Rigidity of Lie algebra homomorphisms). One obtains that if $H^1(\hh,\gg) = 0$, then $\rho$ is rigid. One can also apply Proposition \ref{thm: IFT}, which yields a smooth map that associates to each Lie algebra homomorphism  $\rho'$ close to $\rho$ an element $g \in G$, such that $\Ad_g \circ \rho = \rho'$.
\end{enumerate}

\end{remark}

The following Theorem is Theorem A of \cite{RN3}.

\begin{theorem}[Rigidity of Lie Algebra Homomorphisms] \label{corol: rigidity homomorphism}
Let $\rho:\hh \to \gg$ be a Lie algebra homomorphism. If $H^1(\hh,\gg) = 0$, then there exists a neighborhood $U \subset \hom(\hh,\gg)$ of $\rho$, and a smooth map $h: U \to G$ such that $\rho' = \Ad_{h(\rho')}\circ \rho$, for every Lie algebra homomorphism $\rho' \in U$. In particular, $\rho$ is rigid. 
\end{theorem}

\begin{proof}
Again we set $M = \hom(\hh,\gg)$, $E = M \times \hom(\wedge^2\hh,\gg)$, and we consider the curvature map
\[\mathcal{K}:\hom(\hh,\gg) \longrightarrow \hom(\wedge^2\hh,\gg)\]
defined in the proof of Theorem \ref{theorem:Kuranishi_homomorphisms}, and the $G$ action on $\hom(\hh,\gg)$ via the adjoint representation.

The rigidity problem can be reformulated as \emph{``when is it true that for every $\rho' \in M$ close to $\rho$, and such that $\K(\rho') = 0$, there exists $g \in G$, close the identity, such that $\Ad_g\circ\rho = \rho'$?"}. 

We apply Proposition \ref{thm: IFT} to this situation.
A simple computation shows that the homorphism $\rho$ is a non-degenerate zero of $\mathcal{K}$ iff the sequence
\[\xymatrix{
\gg \ar[r]^-{\delta_{\rho}} & \hom(\hh,\gg) \ar[r]^-{\delta_{\rho}} & \hom(\wedge^2\hh,\gg)},\]
is exact, i.e. iff $H^1(\hh,\gg) = 0$.
\end{proof}

The action of $G$ on $\hom(\hh,\gg)$ factors through an action of the group of automorphisms $\mathrm{Aut}(\gg)$ of the Lie algebra $\gg$, i.e.
$$ \mathrm{Aut}(\gg):= \{A \in \GL(\gg): A\cdot \mu = \mu \}.$$
One can ask for a condition which implies that $\rho: \hh \to \gg$ is rigid {\em with respect to} the action of $\mathrm{Aut}(\gg)$
on the space of Lie algebra homomorphisms.

As a slight variation of Theorem \ref{corol: rigidity homomorphism} we obtain:

\begin{theorem} \label{corol: stable by Aut}
Let $\rho:\hh \to \gg$ be a Lie algebra homomorphism. If $$H^1(\rho^*): H^1(\gg,\gg) \to H^1(\hh,\gg)$$ is surjective, $\rho$ is rigid
with respect to the action of $\mathrm{Aut}(\gg)$.
\end{theorem}

\begin{proof}
Since $\mathrm{Aut}(\gg)$ is a closed subgroup of the Lie group $\GL(\gg)$, it is a Lie subgroup and its tangent space
$\mathrm{aut}(\gg)$ is given by 
$$\{a\in \End(\gg): [a(x),y] + [x,a(y)] = a([x,y]) \quad \forall x,y \in \gg\}.$$
This is exactly the kernel $Z^1(\gg,\gg)$ of $\delta: \hom(\gg,\gg) \to \hom(\wedge^2\gg,\gg)$.

We apply Proposition \ref{thm: IFT}. A simple computation shows that $\rho$ is a non-degenerate zero of
$\K$ (with respect to the action of $\mathrm{Aut}(\gg)$) iff 
\[\xymatrix{
Z^1(\gg,\gg) \ar[r]^-{\rho^*} & \hom(\hh,\gg) \ar[r]^-{\delta_{\rho}} & \hom(\wedge^2\hh,\gg)},\]
is exact, where $\rho^*: \hom(\wedge^k \gg,\gg) \to \hom(\wedge^k \hh,\gg)$ is the chain map defined by
$$ (\rho^*\omega)(v_1,\cdots,v_k) := \omega(\rho(v_1),\cdots,\rho(v_k)).$$
Hence, non-degeneracy of $\rho$ is equivalent to $\rho^*$
mapping $Z^1(\gg,\gg)$ surjectively onto the kernel of $\delta_\rho: \hom(\hh,\gg)\to \hom(\wedge^2\hh,\gg)$.
This condition is equivalent to the surjectivity of the map
$$ H^1(\rho^*): H^1(\gg,\gg) \to H^1(\hh,\gg)$$
induced on cohomology.
\end{proof}

\begin{remark} \rm
If one takes the action by $\mathrm{Aut}(\gg)$ on $\hom(\hh,\gg)$ into account, one obtains
a Kuranishi model where $H^1(\hh,\gg)$ is replaced by $H^1(\hh,\gg) / \Im H^1(\gg,\gg)$.
Observe that also the last statement of Theorem \ref{theorem:Kuranishi_homomorphisms} changes, since
the $G$-orbits are replaced by the larger $\Aut(\gg)$-orbits.
\end{remark}

By applying Proposition \ref{thm: IFT2} we can solve also Problem \ref{problem: stability homomorphisms} (Stability of Lie algebra homomorophisms):

\begin{theorem}[Stability of Lie Algebra Homomorphisms]\label{corol: stability homomorphism}
Let $\rho: \hh \to \gg$ be a homomorphism of Lie algebras. If $H^2(\hh,\gg) = 0$, then $\rho$ is stable. Moreover, in this case the space of Lie algebra homomorphisms from $\hh$ to $\gg$ is locally a manifold of dimension equal to the dimension of $Z^1(\hh,\gg)$. 
\end{theorem}

\begin{proof}
We apply Proposition \ref{thm: IFT2}. Again, let $M = \hom(\hh,\gg)$, and $E$ and $F$ be the trivial vector bundles over $M$ with fiber $\hom(\wedge^2\hh,\gg)$ and $\hom(\wedge^3\hh,\gg)$, respectively.

Denote by $\Lambda$ the space of all Lie algebra structures on $\gg$. Then for each $\mu \in \Lambda$ we obtain a section $\K_{\mu} \in\Gamma(E)$ given by
\[\K_{\mu}(\varphi)(u,v) = \mu(\varphi(u),\varphi(v)) - \varphi([u,v]),\]
and a bundle map
\[\phi_{\mu}: E \to F, \quad \phi_{\mu}(\varphi)(\eta) = \delta_{\mu,\varphi}(\eta),\]
which makes sense even if $\varphi$ is not a homorphism. Note that the map
\[\Lambda \longrightarrow \Gamma(E)\times \Gamma(\hom(E,F)), \quad \mu \longmapsto  (\K_{\mu}, \phi_{\mu})\]
is continuous if we endow the space of sections with the $\mathcal{C}^1$-topology.

Just as in the proof of Theorem \ref{theorem:Kuranishi_homomorphisms}, a straightforward computation shows that 
\[\phi_{\mu}(\varphi)(\K_{\mu}(\varphi)) = 0 \text{ for all } \varphi \in M.\]
In order to apply Proposition \ref{thm: IFT2}, we must check when the linear sequence
\[\xymatrix{T_{\rho}M \ar[r]^{\d^{\mathrm{vert}}_{\rho}\K}& E_{\rho}\ar[r]^{\phi_{\rho}} & F_{\rho}}\]
is exact. But this sequence is precisely
\[\xymatrix{\hom(\hh,\gg) \ar[r]^-{\delta_{\mu,\rho}}& \hom(\wedge^2\hh,\gg)\ar[r]^-{\delta_{\mu,\rho}} & \hom(\wedge^3\hh,\gg),}\]
which is exact if and only if $H^2(\hh,\gg) = 0$.

Thus, by Theorem \ref{thm: IFT2}, for each $\mu' \in \Lambda$ sufficiently close to $\mu$, there exists a $\rho' \in M$ close to $\rho$ such that $\K_{\mu'}(\rho') = 0$. This concludes the proof.
\end{proof}

Theorem \ref{corol: stability homomorphism} was proven by Richardson in \cite{RN5} (Theorem 7.2).

\begin{remark} \rm
The stability problems are more subtle than the rigidity ones. To justify this statement, we present here a straightforward approach to the stability problem for homomorphisms which leads to an infinitesimal condition. However, we are not able to prove that this condition implies stability (and we do not know if it is true or not).

Let $\Lambda$ be the space of all Lie brackets on $\gg$, and let $N$ be the space of all pairs $(\eta, \varphi)$ consisting of a Lie bracket $\eta$ on $\gg$, and a homorphism $\varphi: \hh \to (\gg,\eta)$. Then the stability problem can be rephrased as: \emph{``is the map $\pr_1: N \to \Lambda$ locally surjective in a neighborhood of $(\mu, \rho)$?"}.

Note that $\Lambda$ sits inside $\tilde{\Lambda} = \hom(\wedge^2\gg,\gg)$ as the zero set of the Jacobiator
\[\J:\tilde{\Lambda} \longrightarrow \hom(\wedge^3\gg,\gg),\]
and $N$ sits inside $\tilde{N} = \hom(\wedge^2\gg,\gg)\times \hom(\hh,\gg)$ as the zero set of a map
\[\Phi:\tilde{N} \longrightarrow \hom(\wedge^3\gg,\gg)\times \hom(\wedge^2\hh,\gg)\]
given by $\Phi(\mu, \rho) = (\J(\mu),\K_{\mu}(\rho))$, where
\[\K_{\mu}(\rho)(u,v) = \mu(\rho(u),\rho(v)) - \rho([u,v]).\]

The infinitesimal condition for stability is the surjectivity of the map 
\[\d_{(\mu,\rho)}\pr_1: \ker \d_{(\mu,\rho)}\Phi \to \ker \d_{\mu}\J. \]
We already know that $\d_{\mu}\J = \delta_{\mu}$. If we compute $\d_{(\mu,\rho)}\Phi$ we obtain
\[\d_{(\mu,\rho)}\Phi(\eta,\varphi) = (\delta_{\mu}\eta, \rho^{\ast}\eta + \delta_{\rho}(\varphi)).\]
Thus, the surjectivity amounts to finding for each $\eta \in \hom(\wedge^2\gg,\gg)$ which is $\delta_{\mu}$-closed, a map $\varphi \in \hom(\hh,\gg)$ such that
\[ \rho^{\ast}\eta = \delta_{\rho}(\varphi).\]
In other words, $\d_{\mu,\rho}\pr_1$ is surjective if and only if
\begin{equation}\label{eq: vanishing}
H^2(\rho^{\ast}):H^2(\gg,\gg) \longrightarrow H^2(\hh,\gg)
\end{equation}
is trivial.

However we do not know how to show (or whether it is true) that the vanishing of \eqref{eq: vanishing} implies that $\rho$ is stable. Instead, we have  shown that if $H^2(\hh,\gg) = 0$, then $\rho$ is stable. Note that when $H^2(\gg,\gg) = 0$, it is trivial to verify that $\rho$ is stable, because in this case $\gg$ is rigid by Theorem \ref{corol: rigidity bracket}.
\end{remark}

\subsection{Lie Subalgebras}
Let $\hh \subset \gg$ be a Lie subalgebra of dimension $k$. We begin by giving the Kuranishi description of the local moduli space of Lie subalgebras near $\hh$.

\begin{theorem}\label{theorem:Kuranishi_subalgebras}
Let $\hh \subset \gg$ be a Lie subalgebra of dimension $k$.
There is an open neighborhood $U$ of $0 \in H^1(\hh,\gg/\hh)$ and a smooth map $\Phi: U\to H^2(\hh, \gg/\hh)$
with the following properties:
\begin{enumerate}
 \item $\Phi$ and its derivative vanish at the origin.
 \item The zeros of $\Phi$ parametrize Lie subalgebras of $\gg$ of dimension $k$, with $0$ corresponding to $\hh$.
 \item The union of the $G$-orbits passing through the Lie subalgebras parametrized by $\Phi$ contains an open neighborhood of $\hh$
in the space of all $k$-dimensional Lie subalgebras of $\gg$ (topologized as a subet of the Grassmannian of $\gg$).
\end{enumerate}

\end{theorem}

\begin{proof}
We apply Proposition \ref{theorem:Kuranishi}. Let $M$ be the Grassmannian of $k$-planes in $\gg$, and let $E$ and $F$ be the vector bundles with fiber over $V\subset \gg$ given by $E_V = \hom(\wedge^2V,\gg/V)$ and $F_V= \hom(\wedge^3V,\gg/V)$. Finally, fix an inner product on $\gg$, which -- for each subspace $V \subset \gg$ -- yields a splitting $s:\gg/V \to \gg$ of the exact sequence
\[\xymatrix{0 \ar[r] & V \ar[r] & \gg \ar[r]^-{\pi} & \gg/V \ar[r] &0.}\]
We denote the corresponding projection of $\gg$ on $V$ by $\omega_s = Id_{\gg} - s\circ \pi$.

Consider the action of $G$ on $M$ and on $E$ induced by the adjoint representation, and the section
\[\sigma: M \longrightarrow E, \quad \sigma(V)(u,v) =  \pi([u,v]).\]
Note that a subspace $V$ is a Lie subalgebra iff $\sigma(V)= 0$, and thus the moduli space of interest is $\sigma^{-1}(0)/G$.
 
Next, consider the vector bundle map
\[\phi: E \longrightarrow F\] 
given by
\[\phi(V)(\eta)(u,v,w) = \pi([s(\eta(u,v)),w]) + \eta(\omega_s([u,v]),w) +\text{ cyclic permutations}.\] 
It follows from the Jacobi identity that 
\[\phi(V)(\sigma(V)) = 0 \text{ for all } V\subset \gg.\]

Moreover, by differentiating at $\hh \in M$, we obtain
\[\d_{\hh}^{\mathrm{vert}}\sigma = \delta_{\hh}\]
from where the proof follows.
\end{proof}

The construction of the Kuranishi model for Lie subalgebras is explained in \cite{RN1}, Section 8.

\begin{remark} \rm
\hspace{0cm}
\begin{enumerate}
\item The second jet of $\Phi$ at $0$ is given by the Kuranishi map from Subsection \ref{subsection:algebra_subalgebras}.
\item One can apply Theorem \ref{theorem:Kuranishi_subalgebras} to solve Problem \ref{problem rigidity subalgebra} (Rigidity of Lie subalgebras). In fact, one obtains that if $H^1(\hh,\gg/\hh) = 0$, then $\hh$ is rigid. Moreover, one can apply Proposition \ref{thm: IFT} in order to obtain a smooth map from an open neighborhood of $\hh$ in the Grassmannian of $k$-planes in $\gg$ to $G$ which associates to each Lie subalgebra $\hh'$ close to $\hh$ an element $h(\hh') \in G$, such that $\Ad_{h(\hh')} \circ \hh = \hh'$:  \
\
\end{enumerate}

\end{remark}

\begin{theorem}\label{corol: rigidity subalgebra}
Let $\hh \subset \gg$ be a Lie subalgebra. If $H^1(\hh,\gg/\hh) = 0$, then there exists a neighborhood $U$ of $\hh$ in the Grassmannian of $k$-planes in $\gg$, and a smooth map $h: U \to G$ such that $\hh' = \Ad_{h(\hh')}(\hh)$, for every Lie subalgebra $\hh' \in U$. In particular,  $\hh$ is a rigid subalgebra. 
\end{theorem}

\begin{proof}
Again we let $M$ denote the Grassmannian of $k$-planes in $\gg$, $E\to M$ the vector bundle whose fiber over a subspace $V$ is given by $E_V = \hom(\wedge^2V,\gg/V)$ and consider the section $\sigma$ of $E$ given by
\[\sigma(V)(u,v) = [u,v]_{\gg}\mod V,\]
and the adjoint action of $G$ on $M$.

The rigidity problem can be reformulated as \emph{``when is it true that for every $\hh' \in \tilde{M}$ close to $\hh$, and such that $\sigma(\hh') = 0$, there exists $g \in G$ close the identity map such that $\Ad_g(\hh) = \hh'$?"}.

We apply Proposition \ref{thm: IFT}. A simple computation shows that $\hh$ is a non-degenerate zero of
$\sigma$ iff the sequence
\[\xymatrix{
\gg \ar[r] & \hom(\hh,\gg/\hh) \ar[r]^-{\delta_{\hh}} & \hom(\wedge^2\hh,\gg/\hh)}\]
is exact.
The first map factors through $\delta_{\hh}:\gg/\hh \to\hom(\hh,\gg/\hh)$ and the linear sequence is exact if and only if $H^1(\hh,\gg/\hh) = 0$.
\end{proof}

Theorem \ref{corol: rigidity subalgebra} corresponds to Theorem 11.4 and Corollary 11.5 of \cite{RN4}.

\begin{remark} \rm
The main difference between this problem and the rigidity problem for the inclusion of $\hh$ in $\gg$ is that in the subalgebra problem one does not impose a priori that $\hh'$ is isomorphic to $\hh$. 
\end{remark}

Finally, we apply Proposition \ref{thm: IFT2} to solve Problem \ref{problem: stable subalgebra} (Stability of Lie subalgebras).
This result was established by Richardson in \cite{RN5} (Theorem 6.2).

\begin{theorem} \label{corl: stability subalgebra}
Let $\hh$ be a Lie subalgebra of $\gg$. If $H^2(\hh,\gg/\hh) = 0$, then $\hh$ is a stable subalgebra of $\gg$. Moreover, in this case the space of $k$-dimensional Lie subalgebras of $\gg$ is locally a manifold of dimension equal to the dimension of $Z^1(\hh,\gg/\hh)$. 
\end{theorem}

\begin{proof}
We use the same notation as in the proof of Theorem \ref{theorem:Kuranishi_subalgebras}.

For each Lie bracket $\mu$ on $\gg$, we obtain a section
\[\sigma_{\mu}: M \longrightarrow E, \quad \sigma_{\mu}(V)(u,v) =  \pi\mu(u,v),\]
and a vector bundle map
\[\phi_{\mu}: E \longrightarrow F\] 
given by
\[\phi_{\mu}(V)(\eta)(u,v,w) = \pi\mu(s(\eta(u,v)),w) + \eta(\omega_s(\mu(u,v)),w) +\text{ cyclic permutations}.\] It follows from the Jacobi identity for $\mu$ that 
\[\phi_{\mu}(V)(\sigma_{\mu}(V)) = 0 \text{ for all } V\subset \gg.\]

Moreover, by differentiating at $\hh \in M$, we obtain
\[\xymatrix{\hom(\hh,\gg/\hh) \ar[r]^-{\delta_{\mu,\hh}}& \hom(\wedge^2\hh,\gg/\hh)\ar[r]^-{\delta_{\mu,\hh}} & \hom(\wedge^3\hh,\gg/\hh),}\]
which is exact if and only if $H^2(\hh,\gg/\hh) = 0$.

Thus, Proposition \ref{thm: IFT2} applies to this problem, and for each Lie bracket $\mu'$ on $\gg$ close to $\mu$, there exists a $\hh' \in M$ close to $\hh$ such that $\sigma_{\mu'}(\hh') = 0$. This concludes the proof.
\end{proof}

\end{document}